\documentclass[11pt,letterpaper]{amsart}

\usepackage[T1]{fontenc}
\usepackage{lmodern}
\usepackage[expansion=false]{microtype}
\usepackage{amsmath,amssymb,mathtools}
\usepackage[margin=1.08in]{geometry}
\usepackage{xcolor}
\usepackage[colorlinks=true,linkcolor=blue,citecolor=red,urlcolor=blue]{hyperref}
\usepackage[alphabetic]{amsrefs}

\makeatletter
\renewcommand{\BibLabel}{%
  \Hy@raisedlink{\hyper@anchorstart{cite.\CurrentBib}\hyper@anchorend}%
  [\thebib]%
}
\makeatother

\numberwithin{equation}{section}
\theoremstyle{plain}
\newtheorem{theorem}{Theorem}[section]
\newtheorem{maintheorem}{Theorem}

\newtheorem{proposition}[theorem]{Proposition}
\newtheorem{lemma}[theorem]{Lemma}
\theoremstyle{definition}
\newtheorem*{paunquestion}{Question 38}

\newcommand{\PP}{\mathbf P}
\newcommand{\cX}{\mathcal X}

\newcommand{\cI}{\mathcal I}
\newcommand{\OO}{\mathcal O}
\newcommand{\ddc}{dd^c}
\newcommand{\D}{\Delta}

\hypersetup{
  pdftitle={Failure of local extension of semipositive metrics in a smooth projective family},
  pdfauthor={Xiangsen Qin},
  pdfsubject={Local nonextension of prescribed semipositive singular Hermitian metrics},
  pdfkeywords={singular Hermitian metric, semipositive curvature, metric extension, multiplier ideal, Zariski decomposition}
}

\begin{document}

\title[Failure of semipositive metric extension]
{Failure of local extension of semipositive metrics in a smooth projective family}
\author[X. Qin]{Xiangsen Qin}
\address{Xiangsen Qin: Chern Institute of Mathematics and LPMC, Nankai University \\
Tianjin 300071, China}
\email{qinxiangsen@nankai.edu.cn}
\date{}

\begin{abstract}
We construct a smooth projective family of rational surfaces and an
effective line bundle on its total space that give a negative answer to
P\u{a}un's metric-extension question. The restricted bundle on the central
fibre carries a semipositively curved singular Hermitian metric with
analytic and minimal singularities, but this metric admits no semipositive
extension to any neighbourhood of the whole fibre, even up to singularity
type. The construction uses four-point blow-ups of the plane, with three
centres becoming collinear on the central fibre. This change produces a
nonzero integrable adjoint section on the central fibre, whereas the
corresponding adjoint systems vanish on nearby noncentral fibres. If a
metric extension existed, an $L^2$ extension theorem would extend that
section, contradicting this vanishing.
\end{abstract}

\subjclass[2020]{Primary 32U05; Secondary 32L05, 32U40, 14D06.}
\keywords{Singular Hermitian metric, semipositive curvature, metric extension.}

\maketitle

\section{Introduction}\label{sec:introduction}

In complex geometry, $L^2$ extension theorems use curvature conditions on
Hermitian metrics to extend holomorphic sections from a submanifold to
the ambient space. Singular metrics are useful in this setting because
their singularities encode the integrability conditions imposed on the
sections. Extending the metric itself raises a different question: can semipositive curvature be preserved while agreeing with a
metric prescribed on the submanifold? For a holomorphic line bundle on
the total space of a holomorphic family, this question is particularly
natural when the submanifold is a fibre. P\u{a}un posed it for
pseudoeffective line bundles in the following formulation, recorded by
Dinew, Guedj, and Zeriahi in \cite{DGZ}*{\S6.4.2, Question~38}.

\begin{paunquestion}[P\u aun]
Let $\pi\colon\cX\to\D$ be a smooth proper family of compact K\"ahler
manifolds over a disc $\D\subset\mathbf C$ centred at $0$. Put
$X_0=\pi^{-1}(0)$, let $L$ be a pseudoeffective line bundle on $\cX$,
and let $h_0$ be a singular Hermitian metric with semipositive curvature on
$L_0:=L|_{X_0}$. Does $h_0$ extend to a singular Hermitian metric with
semipositive curvature on $L$?
\end{paunquestion}

To state our answer precisely, we fix the notation and metric conventions.
For $t\in\D$, we write $X_t=\pi^{-1}(t)$ and $L_t=L|_{X_t}$. Our squared-norm convention
is $|e|_h^2=e^{-\varphi}$ in a local holomorphic frame $e$; notation such
as $h=e^{-\varphi}$ always refers to these local weights. Semipositive
curvature means that the local weights are plurisubharmonic (psh). A restriction
is well defined when its weights are not identically $-\infty$ on any
connected component of the fibre. Exact extension means equality of the
restricted weights in compatible frames.

The following theorem gives a negative answer to P\u{a}un's question,
even for a smooth projective family of rational surfaces.
Here projectivity means that the total space admits a closed holomorphic
embedding into a product of projective space with the base, compatible
with the projection to the base.

\begin{samepage}
\begin{maintheorem}\label{thm:main}
There exist a smooth projective morphism $\pi\colon\cX\to\D$ with
rational surface fibres, an effective line bundle $L$ on $\cX$, and a
semipositively curved singular Hermitian metric $h_0$ on $L_0$ with analytic
and minimal singularities, such that the following holds. On no open
neighbourhood $U$ of $X_0$ is there a semipositively curved singular
Hermitian metric $h$ on $L|_U$ whose restriction to $X_0$ is well defined
and whose restricted weights $\psi$ satisfy
\begin{equation}\label{eq:weight-comparison}
  \psi\geq\varphi_0-C
\end{equation}
for a constant $C$, where $\varphi_0$ are the weights of $h_0$ in compatible
local frames. The singular locus of $h_0$ is a smooth rational curve,
along which its divisorial Lelong number is $1/2$.
\end{maintheorem}
\end{samepage}

The failure in Theorem~\ref{thm:main} contrasts with extension results
under stronger positivity assumptions. Coman, Guedj, and Zeriahi
\cite{CGZ13}*{Theorems~B and~B$'$} extend prescribed semipositive singular
metrics from subvarieties of projective manifolds when the ambient line
bundle is ample. Their later work
\cite{CGZ22}*{Theorems~1.1 and~1.4} applies to K\"ahler classes admitting an
adapted Zariski-open Stein cover, including those in the real
N\'eron--Severi space. Collins and Tosatti \cite{CT}*{Theorem~1.1} extend a K\"ahler current with analytic singularities
from a positive-dimensional compact complex submanifold of a compact
K\"ahler manifold, assuming that the ambient class is K\"ahler. In our
example $L_0\cdot\ell=-1$, where $\ell$ denotes the singularity curve of
$h_0$. Thus $L_0$ is not nef, so the positivity assumptions of these
extension theorems are not satisfied.

To compare this result with earlier counterexamples, it is important to
distinguish extension to the whole ambient space from extension to a
neighbourhood of the prescribed submanifold. Hisamoto's counterexamples
\cite{Hisamoto}*{Examples~4.12--4.13, v2} use the one-point
blow-up $\mu\colon Y=\operatorname{Bl}_p\PP^2\to\PP^2$, with
$p\in\PP^2$. Write $H=\mu^*\OO_{\PP^2}(1)$ and let $E$ be the exceptional
curve. In Example~4.12, $B=H+E$ and $S$ is the strict transform of a line
through $p$, hence a ruling fibre of class $H-E$. On a product
neighbourhood of $S$, the bundle $B$ is
$\operatorname{pr}_1^*\OO_{\PP^1}(2)$, where $\operatorname{pr}_1$ projects
onto $S\simeq\PP^1$. Thus every semipositive metric on $B|_S$ extends locally by pullback.
Example~4.13 uses $B=H$ and
the smooth strict transform $S$ of a plane curve with a node at $p$.
The branches give two distinct points of $S\cap E$; the prescribed metric
has a pole at one and is smooth at the other. Since $B|_E$ is trivial, a
global extension induces a psh weight on the compact curve $E$ in a global
trivialization. Finiteness at one intersection excludes the identically
$-\infty$ case, and constancy contradicts the pole at the other. But a
small neighbourhood of $S$ need not contain all of $E$, so this argument
does not establish local nonextension.

Global extension also enters Matsumura's ampleness criterion
\cite{Matsumura}*{Theorem~1.2}: a pseudoeffective line bundle with
nonzero first Chern class on a projective manifold is ample if every
prescribed semipositive singular metric on every subvariety extends
globally. Example~4.1 gives global failure for a big semiample bundle,
while Example~4.2 connects nonextension with restricted volumes and
Zariski decomposition. In contrast, Theorem~\ref{thm:main} concerns one specified
central fibre of a smooth projective family. The total-space line bundle
is effective, and the central metric has analytic and minimal
singularities, yet extension fails on every neighbourhood of the whole
central fibre.

The conclusion is stronger than failure of exact metric extension.
Condition~\eqref{eq:weight-comparison} allows any restriction that is no
more singular than $h_0$. Minimal singularities mean that $h_0$ is no more
singular than any semipositive metric on the same bundle; hence every
semipositive central metric satisfies
$\psi\leq\varphi_0+C_\psi$. Two metrics have the same singularity type when
their local weights satisfy both $\varphi\leq\psi+C$ and
$\psi\leq\varphi+C'$ for suitable constants. Among semipositive central
metrics, \eqref{eq:weight-comparison} is therefore equivalent to equality
of singularity types. Constants may depend on the metric; compactness makes them uniform on $X_0$.

The obstruction is detected by adjoint sections rather than by a direct
comparison of metrics on different fibres. We use a four-point blow-up
in which three centres become collinear only on the central fibre.
This yields a nonzero integrable central adjoint section, while the
corresponding systems vanish on nearby noncentral fibres. If a metric extension satisfying \eqref{eq:weight-comparison}
existed, the comparison would preserve integrability, and an $L^2$
extension theorem \citelist{\cite{Cao}\cite{ZhouZhu2020}} would then
extend the section. Vanishing on nearby
noncentral fibres would then force it to vanish by
holomorphicity, contradicting its nonzero central restriction. Minimal singularities are not needed for this contradiction.

The remainder of this paper is organized as follows.
Section~\ref{sec:obstruction} proves the adjoint criterion.
Section~\ref{sec:construction} constructs the example and establishes
minimal singularities via Boucksom's divisorial decomposition.
Section~\ref{sec:further} computes the integrable adjoint systems:
$m=4$ is the first index at which the criterion applies, not a threshold
for nonextension.

\section{An adjoint criterion for local nonextension}\label{sec:obstruction}

For a metric $h$ on a line bundle $B$, $h^m$ denotes the induced metric
on $B^{\otimes m}$. We write $K_Y$ for the canonical line bundle and
$|D|$ for the complete linear system of a divisor $D$.

We normalize curvature by
\[
 d^c=\frac{\partial-\bar\partial}{4\pi i},\qquad
 \Theta_h=\ddc\varphi=\frac{i}{2\pi}\partial\bar\partial\varphi;
 \qquad \ddc\log|w|^2=[w=0].
\]
Here $\Theta_h$ is the curvature current of $h$ with local weight $\varphi$,
$w$ is a local defining function for a divisor, and $[D]$ denotes
the current of integration along a divisor $D$.
The multiplier ideal $\cI(h)$ consists of holomorphic germs $f$ for which
$|f|^2e^{-\varphi}$ is locally integrable. The ideal thus records the
vanishing needed for local square-integrability with respect to the
singular metric. We regard sections twisted by
$\cI(h)$ as a subspace of untwisted sections via $\cI(h)\hookrightarrow\OO$.
We use additive notation for tensor products of line bundles and write
$K_{\cX/\D}=K_{\cX}\otimes\pi^*K_\D^{-1}$. Divisors also denote their
associated line bundles. Integral divisor identities below are identities
in the Picard group, not merely numerical equivalences; rational
identities are understood in $\operatorname{Pic}(X_0)\otimes\mathbf Q$.
An adjoint bundle is the canonical bundle twisted by another line bundle.

\begin{samepage}
\begin{lemma}[Integrable adjoint extension, Cao]\label{lem:extension}
Let $\pi\colon\cX\to\D$ be a proper holomorphic submersion with K\"ahler
total space. Let $(M,h_M)$ be a line bundle with a singular Hermitian
metric of semipositive curvature whose restriction to $X_0$ is well
defined. Every section
\[
 s_0\in H^0\!\left(X_0,(K_{X_0}+M|_{X_0})
                 \otimes\cI(h_M|_{X_0})\right)
\]
extends to $S\in H^0(\cX,K_{\cX/\D}+M)$ with $S|_{X_0}=s_0$.
\end{lemma}
\end{samepage}

\begin{proof}
This is \cite{Cao}*{Theorem~1.1}; see also
\cite{DemaillyICM}*{Lemma~6.2} for projective families.
For a more general optimal $L^2$-extension theorem from analytic subvarieties
in weakly pseudoconvex K\"ahler manifolds, allowing singular Hermitian
metrics on the twisting line bundle, see Zhou--Zhu \cite{ZhouZhu2020}.
The multiplier
ideal condition and compactness of $X_0$ give the finite central $L^2$
norm required there, without any strict positivity assumption. With $t$
a coordinate on the disc, the theorem extends $s_0\wedge dt$ to
$K_{\cX}+M$. Trivializing $K_\D$ by $dt$ gives the asserted relative
section. In coordinates $(z_1,\ldots,z_n,t)$, where $n$ is the fibre dimension,
and with $e$ a local holomorphic frame of $M$, put $z=(z_1,\ldots,z_n)$.
The fibre restriction is
$f(z,0)\,dz_1\wedge\cdots\wedge dz_n\otimes e|_{X_0}$ when the absolute
section is $f(z,t)\,dz_1\wedge\cdots\wedge dz_n\wedge dt\otimes e$,
with $f$ its holomorphic coefficient.
\end{proof}

\begin{samepage}
\begin{proposition}[Local nonextension criterion]\label{prop:obstruction}
Let $\pi\colon\cX\to\D$ be as in Lemma~\ref{lem:extension}, let $L$ be a
line bundle on $\cX$, and let $h_0$ be a semipositively curved singular
Hermitian metric on $L_0$, with local weights $\varphi_0$.
Suppose that for some integer $m\geq1$,
\[
 H^0\!\left(X_0,(K_{X_0}+mL_0)\otimes\cI(h_0^m)\right)\ne0,
 \qquad H^0(X_t,K_{X_t}+mL_t)=0
\]
for all sufficiently small $t\ne0$. Then no neighbourhood of $X_0$
admits a semipositive metric on $L$ with well-defined restricted weights
$\psi\geq\varphi_0-C$ for a constant $C$.
\end{proposition}
\end{samepage}

\begin{proof}
Suppose such a metric $h$ exists on $U\supset X_0$, and choose a nonzero
central section $s_0$ in the stated subspace. Properness makes
$\pi(\cX\setminus U)$ closed and disjoint from $0$. Hence, for some
$r>0$, $\cX_r:=\pi^{-1}(\D_r)\subset U$, where $\D_r=\{|t|<r\}$; take
$r$ small enough that the asserted vanishing also holds. The comparison
\[
 e^{-m\psi}\leq e^{mC}e^{-m\varphi_0}
\]
shows that $s_0$ is integrable for $(h|_{X_0})^m$.
Lemma~\ref{lem:extension}, applied to $(mL,h^m)$ on $\cX_r$, gives
$S\in H^0(\cX_r,K_{\cX_r/\D_r}+mL)$ with $S|_{X_0}=s_0$.
For every $0<|t|<r$, the vanishing assumption gives $S|_{X_t}=0$.
Thus $S$ vanishes on the open dense set $\cX_r\setminus X_0$, and hence
everywhere, contradicting $s_0\ne0$.
\end{proof}

\section{The four-point construction and proof of Theorem A}
\label{sec:construction}

\subsection{The family and its nearby fibres}

First work over the affine line $\mathbb A^1$ with coordinate $t$.
In $\PP^2\times\mathbb A^1$, with homogeneous coordinates $[x:y:z]$,
consider the four disjoint sections
\begin{align*}
 p_1(t)&=[1:0:0], & p_2(t)&=[0:1:0],\\
 p_3(t)&=[1:1:t], & p_4(t)&=[0:0:1].
\end{align*}
Blow up their union and restrict the resulting algebraic morphism
analytically to a disc $\D$. Denote the restricted blow-up by
\[
 \beta\colon\cX=\operatorname{Bl}_{\bigsqcup_i p_i(\D)}
 (\PP^2\times\D)\longrightarrow\PP^2\times\D,
 \qquad \pi=\operatorname{pr}_2\circ\beta.
\]
Here $\operatorname{pr}_1$ and $\operatorname{pr}_2$ denote the projections
to $\PP^2$ and $\D$, respectively.
Near each centre, relative coordinates $(u,v,t)$ identify its ideal
with $(u,v)$ and the blow-up with
$\operatorname{Bl}_0(\mathbf C^2)\times\D$. Thus $\pi$ is smooth and
projective, with fibres $X_t$ the blow-ups of $\PP^2$ at the four points.
These are rational surfaces.
The algebraic total space over $\mathbb A^1$ is smooth and quasi-projective;
a locally closed embedding into projective space supplies a K\"ahler form
whose analytic restriction makes $\cX$ K\"ahler.

Let $H=(\operatorname{pr}_1\circ\beta)^*\OO_{\PP^2}(1)$, let $E_i$ be
the exceptional divisors over the corresponding sections $p_i(\D)$, and set
\[
 L=2H-E_1-E_2-E_3-E_4.
\]
We use the same notation for their restrictions to a fibre.

\begin{samepage}
\begin{lemma}\label{lem:psef}
The line bundle $L$ is effective on $\cX$.
\end{lemma}

\begin{proof}
The relative quadratic forms
\begin{equation}\label{eq:quadrics}
 F_1=x(z-ty),\qquad F_2=y(z-tx)
\end{equation}
vanish along all four sections and lift to sections of $L$ whose
restrictions to every fibre are not identically zero. Either gives an effective divisor and its
semipositive divisor metric on $L$, also after any shrinking of the base.
\end{proof}
\end{samepage}

For $t\ne0$, solving $F_1=F_2=0$ gives exactly the four points. To check
that no base point (a common zero of the lifted sections) remains on the
exceptional curves, use respectively
\[
\begin{array}{c|c|c}
\text{point}&\text{affine coordinates }(u,v)&
\text{linear terms of }(F_1,F_2)\\ \hline
p_1&x=1,\ (u,v)=(y,z)&(v-tu,-tu)\\
p_2&y=1,\ (u,v)=(x,z)&(-tu,v-tu)\\
p_3&x=1,\ (u,v)=(y-1,z-t)&(v-tu,v)\\
p_4&z=1,\ (u,v)=(x,y)&(u,v)
\end{array}
\]
Each pair is independent when $t\ne0$. The linear system spanned by the two lifted sections therefore
generates $L_t$, so $L_t$ is nef. Nefness here means nonnegative intersection with every irreducible curve. The intersection relations
$H^2=1$, $H\cdot E_i=0$, and $E_i\cdot E_j=-\delta_{ij}$ give
\begin{equation}\label{eq:nearby-intersections}
 L_t^2=0,\qquad K_{X_t}=-3H+\sum_{i=1}^4E_i,
 \qquad K_{X_t}\cdot L_t=-2.
\end{equation}

\begin{lemma}\label{lem:vanishing}
For every $t\ne0$ and every integer $m\geq1$,
$H^0(X_t,K_{X_t}+mL_t)=0$.
\end{lemma}

\begin{proof}
A nonzero section would give an effective divisor $D_t$ with
$D_t\cdot L_t\geq0$ by nefness, contradicting
$(K_{X_t}+mL_t)\cdot L_t=-2$.
\end{proof}

\subsection{The central metric and an integrable adjoint section}

In computations on $X_0$, we abbreviate $p_i(0)$ to $p_i$.
Let $\ell\subset X_0$ be the strict transform of $z=0$, the line through
$p_1(0),p_2(0),p_3(0)$. This smooth rational curve satisfies
\[
 \ell=H-E_1-E_2-E_3,\qquad \ell^2=-2,\qquad L_0\cdot\ell=-1.
\]
Orthogonality of the Zariski positive part to $\ell$ therefore suggests
that the coefficient of $\ell$ in the negative part is $1/2$. Set
\[
 A:=2L_0-\ell=3H-E_1-E_2-E_3-2E_4.
\]

\begin{lemma}\label{lem:bpf}
The linear system $|A|$ is base-point-free: at every point, some global
section of $A$ does not vanish at that point.
\end{lemma}

\begin{proof}
The corresponding cubics on $\PP^2$ are spanned by
\begin{equation}\label{eq:A-generators}
 x^2z,\qquad xyz,\qquad y^2z,\qquad xy(x-y).
\end{equation}
Indeed, double vanishing at $p_4$ removes $z^3,xz^2,yz^2$; vanishing
at $p_1,p_2,p_3$ leaves precisely the displayed span. Their common zeros
are the four points: away from $z=0$ the first three force $x=y=0$,
and on $z=0$ the fourth forces $xy(x-y)=0$.
At $p_1$ in $x=1$ the linear terms include $z,y$; at $p_2$ in $y=1$
they include $z,-x$; at $p_3$ in $x=1$, with $u=y-1$, they include
$z,-u$. Each pair has no common zero on the exceptional line.
At $p_4$ in $z=1$, the quadratic initial terms include $x^2,xy,y^2$,
again with no common zero on that line. After removing the prescribed
exceptional factors, these sections generate $A$ everywhere.
\end{proof}

In particular, $A$ is nef, and
\begin{equation}\label{eq:zariski}
 L_0=\tfrac12A+\tfrac12\ell,\qquad A\cdot\ell=0,\qquad
 A^2=9-1-1-1-4=2,\qquad (A/2)^2=\tfrac12.
\end{equation}
Together with $\ell^2=-2$, these identities give the classical Zariski
decomposition. Here $A/2$ is the nef positive part, and $\ell/2$ is the
effective negative part. B\'ezout also gives $|L_0|=\ell+|H-E_4|$: a conic through
the three collinear points contains their line. Thus the coefficient of
$\ell$ in the fixed part of $|L_0|$ is $1$, whereas its coefficient in
the Zariski negative part is $1/2$.

Let $h_A$ be the smooth semipositive Fubini--Study metric induced
by~\eqref{eq:A-generators}, and $h_\ell$ the divisor metric on $\OO(\ell)$.
Using the line bundle isomorphism $A+\ell=2L_0$, define $h_0$ by
\begin{equation}\label{eq:central-metric}
 h_0^2=h_Ah_\ell,\qquad
 \Theta_{h_0}=\tfrac12\Theta_{h_A}+\tfrac12[\ell]\geq0.
\end{equation}
For a local equation $w=0$ of $\ell$, compatible frames give
$\varphi_0=\frac12\varphi_A+\frac12\log|w|^2$, where $\varphi_A$ is
the local weight of $h_A$.
The metric has analytic singularities, divisorial Lelong number $1/2$
along $\ell$, and no other singularities.
Analytic singularities are locally logarithmic singularities determined
by holomorphic functions.

\begin{lemma}\label{lem:ideal}
For every integer $m\geq1$,
\begin{equation}\label{eq:multiplier-ideal}
 \cI(h_0^m)=\OO_{X_0}\!\left(-\left\lfloor\frac m2\right\rfloor\ell\right).
\end{equation}
\end{lemma}

\begin{proof}
Near $\ell$, $h_0^m$ is $|w|^{-m}$ times a smooth positive factor.
For a germ vanishing to order $k$ along $\ell$, write $r=|w|$ and take
$\varepsilon>0$ sufficiently small. Normal integrability is
\[
 \int_0^\varepsilon r^{2k-m+1}\,dr<\infty
 \quad\Longleftrightarrow\quad
 k>\frac m2-1
 \quad\Longleftrightarrow\quad
 k\geq\left\lfloor\frac m2\right\rfloor.
\]
Divisibility by this power of $w$ is sufficient everywhere and necessary
at nearby points where the remaining factor is nonzero. Away from
$\ell$ the metric is smooth. In particular, $\cI(h_0^4)=\OO_{X_0}(-2\ell)$.
\end{proof}

Let $q_i$ be the strict transform of the line from $p_4(0)$ to $p_i(0)$,
for $i=1,2,3$, and put $Q=q_1+q_2+q_3$. These are the reduced strict
transforms of $y=0$, $x=0$, and $x-y=0$, so $Q$ is effective and
\begin{equation}\label{eq:Q-class}
 Q=3H-E_1-E_2-E_3-3E_4
   =K_{X_0}+4L_0-2\ell=K_{X_0}+2A
\end{equation}
in the Picard group. Let $s_Q$ be the canonical section of $\OO(Q)$.
Fixing the line bundle identifications in~\eqref{eq:Q-class}, we view it as
\[
 \tau_0=s_Q\in H^0(X_0,K_{X_0}+4L_0-2\ell).
\]
Under the inclusion induced by $\OO(-2\ell)\hookrightarrow\OO$, its image is
\begin{equation}\label{eq:central-section}
 \sigma_0=s_\ell^2\tau_0\in H^0(X_0,K_{X_0}+4L_0),\qquad \sigma_0\ne0,
\end{equation}
where $s_\ell$ is the canonical section of $\OO(\ell)$. This image belongs
to the integrable subspace determined by $\cI(h_0^4)=\OO(-2\ell)$:
the local factor $|w|^4$ cancels the pole $|w|^{-4}$ in its squared norm.

Up to a nonzero scalar, $\sigma_0$ is represented by the quintic
$z^2xy(x-y)$, which vanishes to order three at each of the four blow-up
centres and hence defines a section of
$5H-3\sum_{i=1}^4E_i=K_{X_0}+4L_0$.
Its divisor on $X_0$ is $2\ell+Q$, so its vanishing order along $\ell$
is exactly two, as required by the multiplier ideal.

\subsection{Minimal singularities and conclusion}

\begin{lemma}[Minimal singularities]\label{lem:minimal}
The metric $h_0$ has minimal singularities on $L_0$.
\end{lemma}

\begin{proof}
Let $h'$ be any semipositive singular metric on $L_0$, with curvature
$T$ and local weights $\varphi'$. Boucksom's surface characterization
\cite{Boucksom}*{Theorem~4.8} identifies the negative part
in~\eqref{eq:zariski} with the divisorial negative part of $c_1(L_0)$.
The minimal-multiplicity bound \cite{Boucksom}*{\S3.1, after Definition~3.1}
and the Siu decomposition \cite{Boucksom}*{\S2.2.1} give
\[
 \nu(T,\ell)\geq\tfrac12,\qquad T-\tfrac12[\ell]\geq0,
\]
where $\nu(T,\ell)$ is the generic divisorial Lelong number along $\ell$.
Because both metrics are on $L_0$, their local weight differences define
a global $L^1$ function $u=\varphi'-\varphi_0$, defined almost everywhere.
In distributions,
\[
 \ddc u=T-\tfrac12[\ell]-\tfrac12\Theta_{h_A}
          \geq-\tfrac12\Theta_{h_A}.
\]
Thus $u$ has a global quasi-psh representative (locally a sum of a psh
function and a smooth function), bounded above by compactness. The resulting almost-everywhere inequality extends to the
psh weight representatives, giving $\varphi'\leq\varphi_0+C_{h'}$
everywhere. This is the required comparison with every semipositive
metric, and proves minimality.
\end{proof}

\begin{proof}[Proof of Theorem~\ref{thm:main}]
The construction gives a smooth projective family with rational surface
fibres and, by Lemma~\ref{lem:psef}, an effective line bundle on its total
space. The metric~\eqref{eq:central-metric} has the stated semipositivity
and analytic singularities, and Lemma~\ref{lem:minimal} proves its minimal
singularities. Equations~\eqref{eq:central-section}
and~\eqref{eq:multiplier-ideal} give a nonzero integrable central section
for $m=4$, while Lemma~\ref{lem:vanishing} gives nearby adjoint vanishing.
Proposition~\ref{prop:obstruction} therefore excludes the asserted
extension on any neighbourhood containing the whole central fibre.
\end{proof}

\section{The first nonzero integrable adjoint system}\label{sec:further}

\begin{proposition}\label{prop:first-nonzero}
For the metric $h_0$ above, the space
$H^0(X_0,(K_{X_0}+mL_0)\otimes\cI(h_0^m))$ is zero for $m=1,2,3$ and
nonzero for $m=4$.
\end{proposition}

\begin{proof}
By~\eqref{eq:multiplier-ideal}, the relevant line bundles for $m=1,2,3$ are
\begin{align*}
 K_{X_0}+L_0&=-H,\\
 K_{X_0}+2L_0-\ell&=-E_4,\\
 K_{X_0}+3L_0-\ell&=2H-E_1-E_2-E_3-2E_4.
\end{align*}
The first has negative intersection with the nef class $H$; the second
has negative intersection with an ample class. Neither is effective.
A section of the third would give a plane conic through $p_1,p_2,p_3$
with multiplicity at least two at $p_4$. B\'ezout forces $z=0$ to be a
component, leaving a line with multiplicity two at $p_4$, which is
impossible. For $m=4$, the section~\eqref{eq:central-section} is nonzero
and integrable.
\end{proof}

Consequently, this particular integrable adjoint criterion first detects
nonextension at $m=4$. This is not an onset of metric nonextension:
a semipositive extension of $h_0^r$ on $rL$, for any integer $r\geq1$,
would yield one of $h_0$ on $L$ by taking its $r$th root. Semipositive metrics on $L$ itself
already exist by Lemma~\ref{lem:psef}; what fails is extension of the
prescribed central singularity type.

\section*{AI Declaration}

ChatGPT (OpenAI, GPT-6 Astra; model identifier: \texttt{gpt-6-astra})
assisted with exploratory
constructions, calculations, proof checking and preparation of Lean code.
The author conceived and developed the core ideas and proof strategy
and takes responsibility for the arguments, references, and conclusions.

\section*{Computer-assisted verification}

Selected key steps of the construction and nonextension argument have been formally verified in Lean 4, relative to specified external mathematical inputs. These include $L^2$
extension, surface Zariski decomposition, and results concerning
plurisubharmonic weights and positive currents. The identification of
these models with the geometric objects in the paper is not itself fully
formalized. The Lean code is available at
\href{https://github.com/qinxiangsen/failure-of-semipositive-metric-extension/releases/tag/v1.0.0}{the GitHub release v1.0.0}.

\end{document}